\documentstyle{article}

\begin{document}

\newcommand{\fig}[4]{
        \begin{figure}[htbp]
        \setlength{\epsfysize}{#2}
        \centerline{\epsfbox{#4}}
        \caption{#3} \label{#1}
        \end{figure}
        }

\addtolength{\parskip}{1ex}

\def\a{\alpha}
\def\b{\beta}
\def\c{\chi}
\def\d{\delta}
\def\D{\Delta}
\def\e{\epsilon}
\def\f{\phi}
\def\F{\Phi}
\def\g{\gamma}
\def\G{\Gamma}
\def\k{\kappa}
\def\K{\Kappa}
\def\z{\zeta}
\def\th{\theta}
\def\Th{\Theta}
\def\l{\lambda}
\def\la{\lambda}
\def\m{\mu}
\def\n{\nu}
\def\p{\pi}
\def\P{\Pi}
\def\r{\rho}
\def\R{\Rho}
\def\s{\sigma}
\def\S{\Sigma}
\def\t{\tau}
\def\om{\omega}
\def\Om{\Omega}
\def\smallo{{\rm o}}
\def\bigo{{\rm O}}
\def\to{\rightarrow}
\def\E{{\bf Exp}}
\def\ex{{\bf Exp}}
\def\cd{{\cal D}}
\def\rme{{\rm e}}
\def\hf{{1\over2}}
\def\R{{\bf  R}}
\def\cala{{\cal A}}
\def\cale{{\cal E}}
\def\calz{{\cal Z}}
\def\cald{{\cal D}}
\def\calp{{\cal P}}
\def\Fscr{{\cal F}}
\def\cc{{\cal C}}
\def\calc{{\cal C}}
\def\calh{{\cal H}}
\def\calv{{\cal V}}
\def\bk{\backslash}

\def\out{{\rm Out}}
\def\temp{{\rm Temp}}
\def\overused{{\rm Overused}}
\def\big{{\rm Big}}
\def\moderate{{\rm Moderate}}
\def\swappable{{\rm Swappable}}
\def\candidate{{\rm Candidate}}
\def\bad{{\rm Bad}}
\def\crit{{\rm Crit}}
\def\col{{\rm Col}}
\def\dist{{\rm dist}}
\def\mod{{\rm mod }}
\def\supp{{\rm supp}}

\newcommand{\Exp}{\mbox{\bf Exp}}
\newcommand{\var}{\mbox{\bf Var}}
\newcommand{\pr}{\mbox{\bf Pr}}

\newtheorem{lemma}{Lemma}
\newtheorem{theorem}[lemma]{Theorem}
\newtheorem{corollary}[lemma]{Corollary}
\newtheorem{claim}[lemma]{Claim}
\newtheorem{fact}[lemma]{Fact}
\newtheorem{proposition}[lemma]{Proposition}
\newtheorem{observation}[lemma]{Observation}
\newtheorem{prob}[lemma]{Problem}
\newtheorem{question}[lemma]{Question}
\newtheorem{conjecture}[lemma]{Conjecture}
\newtheorem{example}[lemma]{Example}
\newenvironment{proof}{{\bf Proof.}}{\hfill\rule{2mm}{2mm}}
\newenvironment{definition}{\noindent{\bf Definition.}}{}
\newtheorem{remarka}[lemma]{Remark}
\newenvironment{remark}{\begin{remarka}\rm}{\end{remarka}}

\newcommand{\limninf}{\lim_{n \rightarrow \infty}}
\newcommand{\proofstart}{{\bf Proof\hspace{2em}}}
\newcommand{\tset}{\mbox{$\cal T$}}
\newcommand{\proofend}{\hspace*{\fill}\mbox{$\Box$}}
\newcommand{\bfm}[1]{\mbox{\boldmath $#1$}}
\newcommand{\reals}{\mbox{\bfm{R}}}
\newcommand{\expect}{\mbox{\bf Exp}}
\newcommand{\he}{\hat{\e}}
\newcommand{\card}[1]{\mbox{$|#1|$}}
\newcommand{\rup}[1]{\mbox{$\lceil{ #1}\rceil$}}
\newcommand{\rdn}[1]{\mbox{$\lfloor{ #1}\rfloor$}}
\newcommand{\ov}[1]{\mbox{$\overline{ #1}$}}
\newcommand{\inv}[1]{\mbox{$1\over #1 $}}

\def\calp{{\cal P}}
\newcommand{\csp}{{CSP_{n,p}(\calp)}}
\newcommand{\cspc}{{CSP_{n,p=c/n^{k-1}}(\calp)}}
\newcommand{\cspp}{{CSP(\calp)}}
\newcommand{\csph}{{\widehat{\CSP}_{n,p}({\cal P})}}
\def\CSP {{CSP}}
\def\Dom {{\rm Dom}}
\def\supp {{\bf supp}}
\def\Prob {{\bf Pr}}
\def\Ex {{\bf Ex}}
\def\Var {{\bf Var}}
\def\thr {{\rm thr }}

\title{{\Large \bf Sharp thresholds for constraint satisfaction problems and homomorphisms}}

\author{
Hamed Hatami {\em and} Michael Molloy\\
Department of Computer Science\\
University of Toronto\\
Toronto, Canada\\
\{hamed,molloy\}@cs.toronto.edu
}
\maketitle

\begin{abstract}
We determine under which conditions certain natural models of random constraint satisfaction problems
have sharp thresholds of satisfiability.  These models  include graph and hypergraph
homomorphism, the $(d,k,t)$-model, and binary constraint satisfaction problems with domain size three.
\end{abstract}

%\thispagestyle{empty}
%\newpage
%\setcounter{page}{1}
%\pagestyle{plain}

\section{Introduction}
Random 3-SAT and its generalizations have been studied intensively
for the past decade or so (see e.g.\ \cite{abm,
ac6,cs,ef,jsv,nae,cr,dm2,fla,xl1,cm}). One of the most interesting
things about these models, and arguably the main reason that most
people study them, is that many of them exhibit what is called a
{\em sharp threshold of satisfiability}\footnote{Defined formally
below.}, a critical clause-density at which the random problem
suddenly moves from being almost surely\footnotemark[1] satisfiable
to almost surely unsatisfiable.    Most of the work on these
problems is, at least implicitly, an attempt to determine the
precise locations of their thresholds.  At this point, these
locations are known only for a handful of the problems, such as
\cite{nae,cr,dm2,fla,xl1,cm}. Just proving the existence of a sharp
threshold for random 3-SAT was considered a major breakthrough by
Friedgut\cite{ef}.  The vast majority of these generalizations
appear to have sharp thresholds, but there are exceptions which are
said to have coarse thresholds\footnotemark[1].

The ultimate goal of the present line of enquiry is to determine
precisely which of these models have sharp thresholds, but this
appears to be quite difficult; in Section \ref{s3} we show that it
is at least as difficult as determining the location of the
threshold for 3-colourability, something that has been sought after
for more than 50 years (see e.g.\ \cite{er,mmsurv,dm}). A more
fundamental goal is to obtain a better understanding of what can
cause some problems to have coarse thresholds rather than sharp
ones.

Molloy\cite{mm1} and independently Creignou and Daud\'{e}\cite{cd}
introduced a wide family of models for random constraint
satisfaction problems which includes 3-SAT and many of its
generalizations. This permits us to study them under a common
umbrella, rather than one-at-a-time. Molloy determined precisely
which models from this family have any threshold at all (\cite{cd}
provides the same result for those models with
domain\footnotemark[1] size 2). But he left open the much more
important question of which models have sharp thresholds. In this
paper, we begin to address this question.  We answer it for two of
the most natural subfamilies - the so-called
$(d,k,t)$-family\footnotemark[1] (Theorem \ref{tdkt}), and the
family of graph and hypergraph homomorphism problems (Theorem
\ref{thomo}). We also study binary constraint satisfaction problems
with domain size 3.

The standard example of a problem with a coarse threshold is
2-colourability. Here, there is a coarse threshold precisely
because unsatisfiability (i.e. non-2-colourability) can be caused
only by the presence of odd cycles.  Roughly speaking, Friedgut's
theorem\cite{ef} implies that a problem exhibits a coarse
threshold iff unsatisfiability is {\em approximately} equivalent
to having one of a set of unicyclic\footnotemark[1] subproblems.
It is not hard to see that if there are unsatisfiable unicyclic
instances of a problem then that problem exhibits a coarse
threshold (or exhibits no threshold at all). This makes it quite
natural to pose the following rule-of-thumb:

{\bf Hypothesis A:} {\em If a random model from the family in \cite{mm1} is such that:
(a) it exhibits a threshold, and (b) every unicyclic instance is satisfiable,
then that threshold is sharp.}

However, reality is not that simple. \cite{mm1} presents a
counterexample to Hypothesis A; others are presented in this paper.
Nevertheless, the hypothesis holds for certain subfamilies of
models. Creignou and Daud\'{e}\cite{cd} conjectured that Hypothesis
A holds for problems with domain-size two. Special cases of this
conjecture were proven by Istrate\cite{gi} and independently
Creignou and Daud\'{e}\cite{cd2}; the proofs of each paper can be
extended to cover the entire conjecture.  Theorems \ref{tdkt} and
\ref{thomo} in this paper show that Hypothesis A holds for the
$(d,k,t)$-models and for homomorphisms to connected graphs.

In general, coarse thresholds can be caused by
much more subtle and insidious reasons than unsatisfiable unicyclic instances.
In this paper we begin to understand some of these reasons by focusing on the
case where the constraint size is two and the domain size is three (a natural
next step after the well-understood domain-size-two case).  In this paper, we identify
a particular subtle property that must hold whenever Hypothesis A fails (Theorem \ref{t23}).
If we permit either greater domain sizes or greater constraint sizes then this
is no longer true.

\subsection{The random models}
In our setting, the variables of a constraint satisfaction problem (CSP)
all have the same domain of permissable
values, $\{1,...,d\}$, and all constraints will have size $k$,
for some fixed integers $d,k$.
Given a $k$-tuple of variables, $(x_1,...,x_k)$,
a {\em restriction} on $(x_1,...,x_k)$ is a  $k$-tuple of values
$R=(\d_1,...\d_k)$ where  $1\leq\d_i\leq d$ for each $i$.
For each $k$-tuple $(x_1,...,x_k)$, the set of restrictions
on that $k$-tuple is called a {\em constraint}.
The {\em empty constraint} is the constraint which contains
no restrictions. We say that an
assignment of values to the variables of a constraint $C$
{\em satisfies} $C$ if that assignment is not one of the restrictions
in $C$.   An assignment of values to all variables in a CSP {\em satisfies}
that CSP if every constraint is simultaneously satisfied.
A CSP is {\em satisfiable} if it has such a satisfying assignment.

It will be convenient to consider a set of canonical variables
$X_1,...,X_k$ which are used only to describe the ``pattern" of a constraint.
These canonical variables are not variables of the actual CSP.
For any $d,k$ there are $d^k$ possible restrictions and $2^{d^k}$
possible constraints over the $k$ canonical variables. We denote this
set of constraints as $\calc^{d,k}$.
For our random model, one begins by specifying a particular probability
distribution, $\calp$ over $\calc^{d,k}$.  We use $\supp(\calp)$ to
denote the support of $\calp$; i.e. the set of constraints $C$ with
$\calp(C)>0$.
Different choices of $\calp$ give rise to different instances of the
model.

We now define our random models.  The ``$G_{n,M}$'' model,
where the number of constraints is fixed to be $M$, is the most common.
But in this paper,
it will be much more convenient to focus on the ``$G_{n,p}$''
model where each $k$-tuple of variables is chosen independently with probability
$p=c/n^{k-1}$ to receive a constraint.  The two models are, in most respects,
equivalent when $M=(c/k!)n$.  In particular, it is straightforward to show that one
exhibits a sharp threshold iff the other does.

{\bf The $\csp$ Model:} Specify $n,p$ and $\calp$ (typically $p=c/n^{k-1}$
for some constant $c$;
note that $\calp$ implicitly specifies $d,k$).
First choose a random $k$-uniform
hypergraph on $n$ variables where each of the ${n\choose k}$ potential
hyperedges is selected with probability $p$.
Next, for each hyperedge $e$, we choose a constraint on the
$k$ variables of $e$ as follows: we take a random permutation from
the $k$ variables onto $\{X_1,...,X_k\}$ and then we select a random
constraint  according to $\calp$ and map it onto
the $k$ variables.

A property holds {\em almost surely} (a.s) if the limit as
$n\rightarrow\infty$ of it holding is $1$. In \cite{cd,mm1} it was
shown that for every $\calp$, either: (i) $\csp$ is a.s. satisfiable
for every $c\geq0$, (ii) $\csp$ is a.s. unsatisfiable for every
$c>0$, or (iii) there is some $c_1<c_2$ such that $\csp$ is a.s.
satisfiable for every $0\leq c\leq c_1$ and $\csp$ is a.s.
unsatisfiable for every $c>c_2$.  We say that $\csp$ has a {\em
sharp threshold of satisfiability} if there is some positive-valued
function $c(n)=\Theta(1)$ such that for every $\e>0$, if
$p=(1-\e)c(n)/n^{k-1}$ then $\csp$ is a.s. satisfiable and if
$p=(1+\e)c(n)/n^{k-1}$ then $\csp$ is a.s. unsatisfiable.  This is
often abbreviated to just {\em sharp threshold}.  We say that $\csp$
has a {\em coarse threshold} if for all $c$ in some interval
$c_1(n)<c<c_2(n)$,  $\csp$ is neither a.s. satisfiable nor a.s.
unsatisfiable. It is easy to see that every $\calp$ satisfying case
(iii) above must have either a coarse threshold or a sharp
threshold.
%(It is well understood which choices of $\calp$
%have neither type of threshold\cite{mm1,cd}.)

Each $k$-tuple of vertices can have at most one constraint in $\csp$.  When applying
Friedgut's theorem, it will be convenient to relax this condition, and allow
$k$-tuples to possibly receive multiple constraints.  Thus up to
$k!\times|\supp(\calp)|$ constraints can appear
on a $k$-tuple of variables.

{\bf The $\csph$ Model:} Specify $n,p$ and $\calp$. For each of the $n(n-1)...(n-k+1)$
{\em ordered} $k$-tuples of variables and each constraint $C\in\supp(\calp)$,
we assign $C$ to the ordered $k$-tuple with probability $\calp(C)\times p/k!$.

Note that the expected total number of constraints is the same under each model.
Furthermore, it is easy to calculate that the probability of at least one
$k$-tuple receiving more than one constraint in $\csph$ is for $k\geq 3$, $o(1)$
and for $k=2$, an absolute constant $0<\a< 1$.  It follows that if a property
holds a.s. in $\csph$ then it holds a.s. in $\csp$.  As a corollary, we have:

\begin{lemma}\label{l1}
If $\csph$ has a sharp threshold then so does $\csp$.
The reverse is true for $k\geq3$.
\end{lemma}

So for the remainder of the paper, whenever we wish to prove that $\csp$ has a sharp
threshold, we will work in the $\csph$ model.

We often focus on the {\em constraint hypergraph} of a CSP; i.e. the hypergraph
whose vertices are the variables and whose edges are the tuples of variables
that have constraints.  A {\em tree-CSP} is a CSP whose constraint
hypergraph is a hypertree. A CSP is {\em unicyclic} if its constraint
hypergraph is unicylic; i.e. has exactly one cycle. (Hypertree and cycle are defined below).

$F_1$ is said to be a {\em sub-CSP} of $F_2$ if every variable of
$F_1$ is a variable of $F_2$ and every constraint of $F_1$ is a constraint of $F_2$.

We close this subsection with some hypergraph definitions.
A hypergraph consists of a set of vertices and a set of {\em hyperedges},
where each hyperedge is a collection of vertices.  If every hyperedge
has size exactly $k$ then the hypergraph is {\em $k$-uniform}.  In a {\em simple}
hypergraph, no vertex appears twice in any one hyperedge, and no two edges
are identical.  So, for example, the constraint hypergraph of $\csp$ is
simple, but the constraint hypergraph of $\csph$ may have multiple edges.
Neither model permits multiple copies of a vertex in a single edge,
but such edges are possible when we discuss hypergraph homomorphism problems.
The edge $(v,v,...,v)$ is called a {\em loop}.

A walk $P$ of length $r$ is a sequence of $r$ hyperedges and $r+1$
vertices $(v_0, e_1,v_1, e_2, v_2 \ldots, e_r, v_r)$ such that
$e_i$ contains both $v_{i-1}$ and $v_i$.  A walk  is a {\em path}
if the $v_i$ are distinct.  A walk is a {\em cycle} of size $r$ if
for $i=1,\ldots,r$ the $v_i$ and $e_i$ are distinct, and
$v_0=v_r$. The {\em distance} from a vertex $u$ to a vertex $v$ is
the minimum $r$ such that there exists a walk of length $r$,
$(v,e_1,v_1,\ldots,e_r,u)$; the distance of a vertex from itself
is defined to be $0$. The distance from a vertex $v$ to a set of
vertices is the minimum distance from $v$ to any vertex in the
set. A hypergraph is a {\em hypertree} if it has no cycles and it
is connected.

By {\em contracting}
two vertices $u$ and $v$ into a new vertex $w$, we mean (i)
adding a new vertex $w$ to the set of the vertices, (ii) replacing
$u$ and $v$ in every hyperedge by $w$, and (iii) removing $u$ and $v$.
Note that this may result in a hyperedge containing multiple copies
of $w$.

\subsection{Two special families}\label{section:hom}
The ultimate goal of this research is to characterize all
distributions ${\cal P}$ for which $\csp$
exhibits a sharp threshold. However, in section \ref{s3}, we will
show that this is very difficult by proving the following:
\begin{observation}\label{thard}
If one can determine which distributions ${\cal P}$ yield sharp thresholds
for $\csp$ then one can determine the locations of each of those thresholds.
\end{observation}
Determining the locations of some of these thresholds,
eg the 3-SAT threshold or the 3-colourability threshold, is notoriously difficult
(see \cite{mmsurv} for a survey on those two thresholds).
So Observation \ref{thard}
suggests that we should set our sights lower and focus on some important
classes of distributions.

%the following example indicates
%that this is very difficult, even for binary CSP's (the case where $k=2$).
%In particular, it is at least as difficult as determining the location
%of the 3-colourability threshold, a heavily pursued open problem.
%(The existence of that threshold was proven in \cite{af};
%see \cite{mmsurv} for a recent survey and see \cite{ac7,mcd} for the
%best current bounds on its location.)

Perhaps the most natural choice for $\calp$ is the distribution
obtained by selecting each of the $d^k$ possible restrictions
independently with probability $q$ for some constant $q$.
However, as noted in \cite{akk}, every such
choice of $\calp$ yields a model that is a.s. unsatisfiable for
any non-trivial choice of $p$.  So this is a rather uninteresting
family of models, particularly as far as the study of thresholds
goes.

The next most natural choice for $\calp$ is to fix $t$, the number of restrictions
per clause, and to make every constraint with exactly $t$ restrictions equally likely.
(Note that for $d=2,t=1$ this yields random $k$-SAT.)
This is often called the $(d,k,t)$-model and has received a great deal of study,
both from a theoretical perspective \cite{mit,ms} and from experimentalists
(see \cite{gmp} for a survey of many such studies).  In \cite{akk} it is shown
that when $t\geq d^{k-1}$,  this model is problematic in the same way as
the previously mentioned one, as it is a.s. unsatisfiable even for values
of $p=o(1/n^{k-1})$ (i.e. when the number of constraints is $o(n)$).
However, it was proven in \cite{gmp} that for every $1\leq t<d^{k-1}$, the
$(d,k,t)$-model does not have that problem.  One of the main contributions of
this paper is to show that for this case, the model  exhibits a sharp threshold:

\begin{theorem} \label{tdkt} For every $d,k\geq2$ and every $1\leq t< d^{k-1}$,
the $(d,k,t)$-model has a sharp threshold.
\end{theorem}

Note that this generalizes the well-known result that $k$-SAT has a sharp
threshold (\cite{cr,ag} for $k=2$; \cite{ef} for $k\geq3$), as can
be seen by setting $d=2,t=1$.

From a different perspective, it is quite natural to consider the case where
every constraint is identical, i.e. $|\supp(\calp)|=1$.  It is not hard
to see that every such problem
is equivalent to a hypergraph homomorphism problem, as defined below:

For two $k$-uniform hypergraphs, $G,H$, a {\em homomorphism} from
$G$ to $H$ is a mapping $h$ from $V(G)$ to $V(H)$  such that for
each edge $(v_1,v_2,\ldots,v_k)$ of $G$,
$(h(v_1),h(v_2),\ldots,h(v_k))$ is an edge of $H$. We say that $G$
is {\em homomorphic} to $H$, if there exists such a homomorphism.
When $k=2$ and $H$ is the complete graph with no loops, we are
simply asking whether $G$ has a $|H|$-colouring. Homomorphisms are
an important generalization of graph colouring (see e.g. \cite{hn}).
They are often also referred to as $H$-colourings (e.g.
\cite{hn2,gkp}).

Suppose that $H$ is a fixed undirected $k$-uniform hypergraph, and $G$ is a random
$k$-uniform hypergraph on $n$ vertices where each of the ${n\choose k}$
potential hyperedges is selected with probability $p$.  Set $d$ to
be equal to the number of vertices in $H$ and define a constraint
$C$ with domain size $d$ and constraint size $k$ by saying that
$C$ permits $X_1=\d_1,...,X_k=\d_k$ iff $(\d_1,...,\d_k)$ is a
hyperedge of $H$.  Treat each vertex of $G$ as a variable with
domain $\{1,..,d\}$ and assign $C$ to each hyperedge of $G$. We
call this the {\em $H$-homomorphism problem}.

Thus we have an instance of $\csp$ where $C$ is the only
constraint in $\supp(\calp)$ and furthermore $C$ is symmetric
under permutations of the canonical variables; in other words, all
constraints are identical even under permutations of variables. It
is easy to see that every such $\calp$ corresponds to a
homomorphism problem; just take $H$ to be the hypergraph where
$(\d_1,...,\d_k)$ is a hyperedge iff $C$ permits
$X_1=\d_1,...,X_k=\d_k$.  Note that here a hyperedge in $H$ may
contain multiple copies of a vertex.

Thus, these $H$-homomorphism problems are not only important
as a fundamental graph problem, but also because they form a very natural subclass of our
family of random CSP models.
In this paper, we prove that Hypothesis A holds for every connected undirected $H$.

It is easy to see that if $H$ has a loop $(\d,\d,...,\d)$
then every hypergraph  is trivially homomorphic to $H$
(just map every vertex to $\d$); so the $H$-homomorphism problem has no
threshold at all.  The other trivial case is where $H$ has
no hyperedges at all and so no non-trivial hypergraph
has an $H$-homomorphism.

\begin{lemma} \label{cycle-hom}Suppose that $H$ is a nontrivial $k$-uniform hypergraph with
no loops. We have the following:

\begin{enumerate}
\item[(a)] For $k\geq 3$, every unicylic $k$-uniform hypergraph is homomorphic to a
single hyperedge, and hence to $H$.
\item[(b)] For $k=2$: if the triangle is homomorphic to $H$, then so is
every unicyclic graph; and the triangle is homomorphic to $H$ iff
$H$ contains a triangle.
\end{enumerate}
\end{lemma}
\begin{proof}
To prove part (1), let $(v_0, e_1,v_1, e_2, v_2 \ldots, e_r, v_0)$
be the unique cycle of the hypergraph, and let
$(w_0,\ldots,w_{k-1})$ be a single hyperedge. Define
$h(v_i)=w_{(i\ \mod \ 2)}$, for every $0 \le i \le r-2$ and
$h(v_{r-1})=w_{2}$. It is easy to see that one can extend $h$ to a
homomorphism from the unicyclic hypergraph to $H$.

Part (2) easily follows from the easy and well-known fact that
every cycle is homomorphic to the triangle, and the triangle is
not homomorphic to any cycle of size greater that $3$.
\end{proof}

From Lemma~\ref{cycle-hom} we conclude that proving that
Hypothesis A holds whenever $H$ is connected and undirected is
equivalent to proving:

\begin{theorem}\label{thomo} If $H$ is a connected undirected loopless $k$-uniform hypergraph
with at least one edge,
then the $H$-homomorphism problem has a sharp threshold iff (a) $k\geq 3$
or (b) $k=2$ and $H$ contains a triangle.
\end{theorem}

Note that this generalizes the well-known result that $c$-colourability
has a sharp threshold for $c\geq3$\cite{af}, as can be seen by setting
$k=2$ and taking $H=K_c$.

We do not have a strong feeling
as to whether the ``connected'' condition is necessary here; we discuss
the possibility of extending Theorem \ref{thomo} to disconnected graphs in
Section \ref{sh}. In section \ref{sdg}, we provide a disconnected directed graph $H$
for which the $H$-homomorphism problem is a counterexample to the analogue
of Hypothesis A in the $\csph$ model but not in the $\csp$ model.

\subsection{Tools}
Our main tool is distilled from Friedgut's main theorem in
\cite{ef}.  Friedgut reported to us[private communication] that his
proof can be adapted to the setting of this paper; for completeness,
the first author worked out the details of such an extension in
\cite{htr} as they did not appear in print. To provide Friedgut's
theorem for CSP's in its full power instead of being restricted to
the unsatisfiability property, we consider, as Friedgut did,
properties that are preserved under constraint addition; such
properties are called \emph{monotone}. A property on CSP's is called
\emph{monotone symmetric} if it is monotone and invariant under CSP
automorphisms. For a property $A$, $A_n$ denotes the restriction of
$A$ on CSP's with exactly $n$ variables.  Roughly speaking,
Friedgut's theorem says that for a value of $p$ that is ``within''
the coarse threshold, there is a constant sized instance $M$ such
that $\tau<\Pr[M \subseteq \widehat{\CSP}_{n,p}({\cal P})]<1-\tau$
for some constant $\tau$ which does not depend on $n$, and adding
$M$ to our random CSP boosts the probability of being in $A$ by at
least $2\a>0$, whereas adding a linear number of new random
constraints only boosts it by at most $\a$. First, we must formalize
what we mean by ``adding $M$''. Given two CSP's $M,F$ where $M$ has
$r$ variables, and $F$ has at least $r$ variables, we define
$F\oplus M$ to be the CSP obtained by choosing a random $r$-tuple of
variables in $F$ and then adding $M$ on those $r$ variables. Now we
can formally state an adaptation of Friedgut's theorem to the
setting of this paper. A proof can be found in \cite{htr}:

\begin{theorem}
\label{appendix:toolcsp} Let $A$ be a monotone
symmetric property in $\widehat{\CSP}_{n,p}({\cal P})$ which has a
coarse threshold. There exist $p=p(n)$, $\tau,\alpha,\epsilon>0$,
and a CSP $M$ whose constraints are chosen from $\supp({\cal P})$ such
that for an infinite number of $n$:

\begin{itemize}
\item[(a)] $\alpha<\Pr[\widehat{\CSP}_{n,p}({\cal P}) \in
A]<1-2\alpha$.

\item[(b)] $\Pr[\widehat{\CSP}_{n,(1+\epsilon)p}({\cal P}) \in A]
< 1- 2\alpha$.

\item[(c)] $\Pr[\widehat{\CSP}_{n,p}({\cal P}) \oplus M \in A]
> 1- \alpha$.

\item[(d)] $\tau<\Pr[M \subseteq \widehat{\CSP}_{n,p}({\cal
P})]<1-\tau$.

\end{itemize}
\end{theorem}

Since in our setting $p(n)=c(n)/n^{k-1}$ where $c(n)=\theta(1)$,
Theorem~\ref{appendix:toolcsp}(d) implies that $M$ has exactly one
cycle.
%Corollary \ref{cmm} below,
%implies that $M$ is not a tree. So we obtain the following corollary which is our main tool
%in this paper.

\begin{corollary}
\label{toolcsp} For any $\calp$, if $\csph$ has a coarse
threshold of satisfiability then there exist $p=p(n)$,
$\alpha,\epsilon>0$, and a unicyclic CSP $M$ on a constant number of variables whose
constraints are chosen from $\supp(\calp)$ such that:

\begin{itemize}
\item[(a)] $\alpha<\Pr[\widehat{\CSP}_{n,p}({\cal P}) \mbox{ is unsatisfiable}]<1-2\alpha$.
\item[(b)] $\Pr[\widehat{\CSP}_{n,(1+\epsilon)p}({\cal P}) \mbox{ is unsatisfiable}] < 1- 2\alpha$.
\item[(c)] $\Pr[\widehat{\CSP}_{n,p}({\cal P}) \oplus M \mbox{ is unsatisfiable}]> 1- \alpha$.
\end{itemize}
\end{corollary}

%{\bf Remark:}  It would have been more convenient if we could have applied
%Bourgain's Theorem (see \cite{ef}) rather than Friedgut's as it
%already applies to our setting and so there would be no need for an adaptation.
%But Bourgain's Theorem does not imply that $M$ is unicyclic, only that
%it is satisfiable.  The fact that it is unicylic is crucial to our
%arguments in Section \ref{sh}.

Our next tool proves some properties for local parts of a random CSP.

\begin{lemma}
\label{local} Suppose that $p<cn^{1-k}$ for some positive constant
$c$, and let $G$ be an instance of $\widehat{\CSP}_{n,p}({\cal
P})$. Let $t$ be a positive constant integer and choose a set $T$ of $t$ random variables. Then for every
$\epsilon>0$, and integer $r>0$ there exists an integer
$L(c,t,r,\epsilon)$ such that with  probability at least
$1-\epsilon$:
\begin{itemize}
\item[(i):] No constraint of $G$ contains more than one variable of $T$.
\item[(ii):] $G$
induces a forest on the set of the variables that are of distance
at most $r$ from $T$. \item[(iii):] There are at most $L$
variables that are of distance at most $r$ from $T$.
\end{itemize}
\end{lemma}
\begin{proof}
Let $E_1$, $E_2$, and $E_3$ denote the events $(i)$, $(ii)$, and
$(iii)$ respectively. Trivially
\begin{equation}
\label{eq1} \Pr[E_1] \ge 1 - \sum_{i=2}^k n^{k-i}{t \choose i}k! p=1-o(1).
\end{equation}
The expected number of the cycles of size at most $2r$ which
contain at least one variable in $T$ is at most $t\sum_{i=2}^{2r}
n^{ik-i-1}p^{i}$. Thus
\begin{equation}
\label{eq2} \Pr[E_2] \ge 1-t\sum_{i=2}^{2r} n^{ik-i-1}p^{i} \ge
1-\frac{2tr(1+c)^{2r}}{n}=1-o(1).
\end{equation}

The expected number of the variables in a distance of at most $r$
from $T$ is at most $t\sum_{i=1}^{r}n^{ik-i}p^i$. So by
Chebychev's inequality, for sufficiently large $L$:
\begin{equation}
\label{eq3} \Pr[E_3] \ge 1-\frac{t\sum_{i=1}^{r}n^{ik-i}p^i}{L}
\ge 1-\frac{\epsilon}{2}.
\end{equation}
The lemma follows from (\ref{eq1}), (\ref{eq2}), and (\ref{eq3}).
\end{proof}

Our third tool is easily proven with a straightforward first moment
calculation and concentration argument (via e.g. the second moment
method or Talagrand's inequality); we omit the details.

\begin{lemma}
\label{tree} Let $T$ be a tree-CSP whose constraints are in
$\supp({\cal P})$. There exists $z=o(n^{1-k})$ such that a.s.
$\widehat{\CSP}_{n,z}({\cal P})$ contains $T$ as a sub-CSP.
\end{lemma}

%\begin{corollary}\label{cmm}
%In the setting of this paper, the subgraph $M$ from Theorem \ref{appendix:toolcsp}
%is not a tree.
%\end{corollary}

%\begin{proof} Take $p,\a,M$ as in Theorem \ref{appendix:toolcsp};
%in the setting of this paper, $p=\Theta(n^{1-k})$.
%If $M$ is a tree, then by Lemma \ref{tree}, there exists $z=o(n^{1-k})$
%such that adding a copy of $\widehat{\CSP}_{n,z}({\cal P})$ to $\widehat{\CSP}_{n,p}({\cal P})$
%a.s. results in the addition of a random copy of $M$ to $\widehat{\CSP}_{n,p}({\cal P})$.
%The former addition is equivalent to simply taking $\widehat{\CSP}_{n,p+z}({\cal P})$.
%Therefore, Theorem \ref{appendix:toolcsp}(c) implies that
%$\pr[\widehat{\CSP}_{n,p+z}({\cal P})]\in A>1-\a$.  Since $(1+\e)p>p+z$ (for sufficiently large $n$),
%this contradicts Theorem \ref{appendix:toolcsp}(b).
%\end{proof}

\section{Difficulty}\label{s3}
Here we prove Observation \ref{thard}, thus showing that characterizing those
distributions $\calp$ for which $\csp$ has a sharp threshold is at least as difficult as determining the location
of all such thresholds.

Suppose, for example, that one wanted to know the 3-colourability
threshold; i.e. the value $c=c(n)$ such that for all $\e>0$,
$\chi(G_{n,p=c(n)-\e})$ is a.s. 3-colourable and
$\chi(G_{n,p=c(n)+\e})$ is a.s. not 3-colourable. (The existence of
this threshold was proven in \cite{af}.) We will construct a family
of distributions $\calp$ such that determining which of those
distributions have sharp thresholds is sufficient to determine the
3-colourability threshold.

We set $d=5$ and $k=2$, and define  two constraints by listing
their  pairs of forbidden values:
\begin{eqnarray*}
C_1&=&\{(4,4),(5,5)\} \cup (\{1,2,3\}\times\{4,5\})\cup (\{4,5\}\times\{1,2,3\}),\\
C_2&=&\{(1,1),(2,2),(3,3)\} \cup (\{1,2,3\}\times\{4,5\})\cup (\{4,5\}\times\{1,2,3\}).
\end{eqnarray*}
Note that each constraint forces the endpoints of every edge to take values that are either
both in $\{1,2,3\}$ or both in $\{4,5\}$.  A $C_1$ constraint says that they
have to be different values if they are both in $\{4,5\}$.  A $C_2$ constraint says
that they have to be different values if they are both in $\{1,2,3\}$.

We let $C_1$ occur with probability $q$ and $C_2$ occur with
probability $1-q$ in ${\cal P}$.
Set $c(q)=(1-q)/q$.

\begin{fact}\label{f3c}
\begin{enumerate}
\item[(a)] If $G_{n,p=c(q)/n}$ is a.s. 3-colourable, then $\csp$ has a sharp threshold.
\item[(b)] If there is some $\e>0$ such that $G_{n,p=(c(q)-\e)/n}$
is a.s. not 3-colourable, then $\csp$ has a coarse threshold.
\end{enumerate}
\end{fact}

Thus, determining the type of the
threshold for all such models
$\CSP_{n,p}({\cal P})$ requires
the knowledge of for which values of $c$, $G(n,\frac{c}{n})$ is a.s.
$3$-colourable, and for which values it is a.s. not $3$-colourable.

\proofstart
Choose our $\csp$ by first taking $G_{n,p=c/n}$ and
then setting each edge to be $C_1$ with probability $q$ and $C_2$ otherwise.
Let $G_1,G_2$ be the subgraphs formed by the edges chosen to be $C_1,C_2$ respectively.
If $c<1$ then all components of $G_{n,p=c/n}$ are trees or unicycles
and the CSP is trivially satisfiable.  So we
can focus on the range $c>1$ and we let $T$ denote
the giant component of $G_{n,p=c/n}$.  Note that
the variables of $T$ must either all take values from $\{1,2,3\}$ or
all take values from $\{4,5\}$.

{\em Case 1: $c>\inv{q}$.} Then $G_1$ is equivalent to $G_{n,p=c_1/n}$
for $c_1=cq>1$ and it follows easily that a.s. $G_1$ contains
a giant component which is not 2-colourable. This giant component is a subgraph
of $T$ and so the variables of $T$ must all take values from $\{1,2,3\}$.
It follows that the CSP is satisfiable iff $G_2$ is 3-colourable. Note
that  $G_2$ is equivalent to $G_{n,p=c_2/n}$
for $c_2=c\times(1-q)>c(q)$.

{\em Case 2: $c<\inv{q}$.}
%Suppose $c=\inv{q}-\d$.
Then $G_1$ is equivalent to $G_{n,p=c_1/n}$
for $c_1=cq<1$ and  $G_2$ is equivalent to $G_{n,p=c_2/n}$
for $c_2=c\times(1-q)<c(q)$.  If $G_2$ is a.s. 3-colourable then the CSP
is a.s. satisfiable. If $G_2$ is a.s. not 3-colourable then
the CSP is satisfiable iff $T$ is 2-colourable; i.e., if
$G_1$ does not have an odd cycle lying within $T$. Since $c_1<1$, $G_1$ is a random
graph with edge-density below the critical point.  So with probability at least some
positive constant, it has no odd cycle at all, let alone one lying within $T$. On the other hand,
the distribution of the number of triangles in $G_1$ is asymptotically
Poisson with mean $c_1^3/6$ (see, eg Section 3.3 of \cite{jlr}), and so the probability
of containing at least one triangle tends to $1-\e^{-c_1^3/6}$.  If we condition
on $u,v,w$ forming a triangle in $G_1$,
then the probability that they are all in the giant component
$T$ is easily seen to not increase, and so is at least $(|T|/n)^3$ which tends to a positive
constant. Therefore, the probability that $G_1$ has an odd cycle in $T$
is at least some positive constant.  This implies that the CSP is neither a.s. satisfiable nor
a.s. unsatisfiable.

Fact \ref{f3c} now follows.  If $G_{n,p=c(q)/n}$ is a.s. 3-colourable, then $\csp$ has a sharp threshold
which lies somewhere above $\inv{q}$. If there is some $\e>0$ such that $G_{n,p=(c(q)-\e)/n}$
is a.s. not 3-colourable, then for all $\inv{q}-\frac{\e}{1-q}<c<\inv{q}$,
we are in Case 2 where $c_2>c(q)-\e$ and so $G_2$ is a.s. not 3-colourable.
In that range of $c$, $CSP_{n,p=c/n}(\calp)$ is neither
a.s. satisfiable nor a.s. unsatisfiable. So
$\csp$ has a coarse threshold running from $\inv{q}-\d$ to $\inv{q}$
for some $\d>\e/(1-q)$.
\proofend

It is straightforward to adapt this example so that, instead of 3-colourability,
we use any model $\CSP_{n,p}({\cal P})$ which has a sharp threshold. Suppose that $\calp$ is over
constraints of size $k$ with domain-size $d$. We will create a distribution
$\calp'$ over constraints of size $k$ and with domain-size $d+2$.  All constraints will enforce:
\begin{enumerate}
\item[(1)] All $k$ variables take values in $\{1,...,d\}$ or all $k$ variables
take values in $\{d+1,d+2\}$.
\end{enumerate}
Constraint $C^*$ also enforces:
\begin{enumerate}
\item[(2)] The first two variables cannot both be $d+1$ and they cannot both be $d+2$
\end{enumerate}
For each constraint $C_i\in\supp(\calp)$ we have a constraint $C_i'$ which
has the same restrictions as $C_i$, and also enforces (1).
We set $\calp'(C_i')=(1-q)\calp(C_i)$ and we set $\calp'(C^*)=q$.

Note that if $\calp$ is simply the 3-colouring CSP then this yields the example
from above. Similar reasoning to that above shows that we cannot know which values of
$q$ yield a sharp threshold for $\calp'$ without knowing the location of the threshold for $\calp$.
This yields
Observation \ref{thard}.

\section{Homomorphisms}\label{sh}

In this section, we prove Theorem~\ref{thomo} which concerns
$H$-homomorphisms. Let $G^k_{n,p}$  denote the random $k$-uniform
hypergraph on $n$ vertices where each $k$-tuple is present as a
hyperedge with probability $p$.

We begin with a technical lemma:

\begin{lemma}
\label{distance} Let $H$ be a connected graph which contains a
triangle. Let $u$ be a vertex of $H$ and $M$ be a unicyclic graph
with unique cycle $C$. Denote the vertices of $M$ in a distance of
exactly $r \ge |V(H)|+3$ from $C$ by $U$. There is a homomorphism
from $M$ to $H$ such that all vertices in $U$ are mapped to $u$.
\end{lemma}
\begin{proof}
Let $h$ be a homomorphism from $C$ to the triangle $(v_1,v_2,v_3)$
of $H$. Observe that for $i=1,2,3$, there exist walks
$(v_i=)v_{i,0},\ldots,v_{i,r}(=u)$ of length exactly $r$ in $H$.
Let $w$ be a vertex in $M$ in the distance of $j \le r$ from $C$,
and $w'$ be the vertex of $C$ which has the distance $j$ from $w$.
Extend $h$ by assigning $h(w)=v_{ij}$ where $h(w')=v_i$. Observe
that $h$ is a partial homomorphism from $M$ to $H$ which maps
every vertex in $U$ to $u$. Trivially $h$ can be extended to a
homomorphism from $M$ to $H$.
\end{proof}

{\bf Proof of Theorem \ref{thomo}.} Let $H$ be some $k$-uniform hypergraph, and assume that
the $H$-homomorphism problem has a coarse threshold.  Let $M,p,\a,\e$ be as guaranteed by
Corollary~\ref{toolcsp}.  In this setting, $M$ is a
$k$-uniform unicyclic hypergraph, such that adding $M$ to $G^k_{n,p}$ boosts
the probability of not having a homomorphism to $H$ by at least
$2\a$.

Our strategy will be to show that adding $M$ to $G^k_{n,p}$ does not
increase the probability of not having a homomorphism to $H$ by more
than adding a copy of $G^k_{n,z}$ for some $z(n)=o(n^{1-k})$.  We
are assuming that the former boosts that probability to at least
$1-\a$ and thus so must the latter.  But that will contradict
Corollary \ref{toolcsp}(b).  To show this, we will construct a
hypertree $T$ such that the probability that $G^k_{n,p} \oplus M$
has no homomorphism to $H$ is at most the probability that
$G^k_{n,p} \oplus T$ has no homomorphism to $H$.  Then we simply
apply Lemma \ref{tree} to obtain our desired $z$.

We begin with the case $k\geq3$.

Consider $G=G^k_{n,p} \oplus M$. Let $M^+$ be the subgraph of $G$
consisting of all hyperedges that contain at least one vertex of
$M$ (and, of course all vertices in those hyperedges); in other
words, $M^+$ is the subhypergraph induced by the vertices of $M$
and all their neighbours. Lemma~\ref{local} implies that there is
some constant $L$ such that, defining $E$ to be the event that
``$M^+$ is unicyclic and has at most $L$ vertices,
no hyperedge of $G^k_{n,p}$ contains more than one vertex of $M$,
and the vertices of distance at most 2 from $M^+$ induce a forest'',
we have $\pr(E)\geq 1-\frac{\a}{2}$.

Since $M$ is unicylic and $k\geq3$, by Lemma~\ref{cycle-hom}(a) there exists a
homomorphism $h$ from $M$ to a single edge, say
$(v_1,\ldots,v_k)$. Let $h_i$ be the set of the vertices in $M$
that are mapped by $h$ to $v_i$. Obtain the hypergraph $G'$ from
$G$ by (i) removing all edges in $M$; (ii) contracting all of the
vertices in $h_i$ into one single new vertex $u_i$, for each $1
\le i \leq k$; (iii) adding the single hyperedge
$(u_1,\ldots,u_k)$.

Suppose that $h'$ is a homomorphism from $G'$ to $H$. Then a
mapping from the vertices of $G$ to the vertices of $H$ which maps
every vertex $v$ in $G-M$ to $h'(v)$, and every vertex in $h_i$
to $h'(u_i)$ is a homomorphism from $G$ to $H$.  Thus, if $G'$ is
homomorphic to $H$ then so is $G$.

We define our hypertree $T$ as follows: $T$ has a hyperedge
$(t_1,...,t_k)$, and each $t_i$ lies in $L$ other hyperedges.
Only $t_1,...,t_k$ lie in more than one edge of $T$.  Thus, $T$
has $k+k(k-1)L$ vertices and $kL+1$ hyperedges. If $E$ holds, then
the subgraph of $G'$ induced by all edges containing $\{u_1,\ldots,u_k\}$
forms a tree; note that it is a subtree of $T$. It follows that $G^k_{n,p}\oplus T$ is at
least as likely to be non-homomorphic to $H$ as $G'$ is, so:
\begin{eqnarray*}
\Pr[G^k_{n,p} \oplus M \mbox{ is not homomorphic to $H$}]
&\le &\Pr[G^k_{n,p} \oplus M \mbox{ is not homomorphic to $H$}|E] +\pr(\ov{E})\\
&\le &\Pr[G^k_{n,p} \oplus T \mbox{ is not homomorphic to $H$} ] + \frac{\alpha}{2}.
\end{eqnarray*}
By Lemma~\ref{tree}, there is some $z=o(n)$ such that increasing $p$ by an additional $z$ a.s. results
in the addition of a copy of $T$. Thus for every $\e>0$:
\[\Pr[G^k_{n,p} \oplus T \mbox{ is not homomorphic to $H$}] \le
\Pr[G^k_{n,(1+\epsilon)p}\mbox{ is not homomorphic to $H$}]\] which
yields a contradiction to Corollary~\ref{toolcsp}(b).

This proves the case where $k\geq 3$, so we now turn to the case
$k=2$.  If $H$ contains no triangle, then $K_3$ is not homomorphic
to $H$.  Thus, $K_3$ forms a unicyclic unsatisfiable CSP using the
$H$-colouring constraints and so we do not have a sharp threshold,
since for any $0<c<\hf$:(i) with probability at least some positive constant,
$G_{n,p=c/n}$ is a forest, and hence
has a homomorphism to $H$;
and (ii) with probability at least some positive constant, $G_{n,p=c/n}$
contains a triangle and hence has no homomorphism to $H$.
So we will focus on graphs $H$ that contain a triangle.  Our proof
follows along the same lines as the case $k\geq 3$, but is
complicated a bit since we can no longer assume that $M$ is
homomorphic to a single edge. We only highlight the differences.

Define $M^+$ to be the subgraph of $G=G^k_{n,p}\oplus M$ induced
by all vertices within distance $r=|V(H)|+|V(M)|+3$ of the unique
cycle of $M$.  By Lemma~\ref{local} there is some constant $L$
such that with probability at least $1-\frac{\a}{2}$: $M^+$ is
unicyclic and has at most $L$ vertices, no hyperedge of $G^k_{n,p}$ contains
more than one vertex of $M$,
and the vertices of distance at most 2 from $M^+$ induce a forest.

Define $U$ to be the set of vertices of $G$ that are of distance
exactly $r=|V(H)|+|V(M)|+3$ from the unique cycle of $M$. Consider
any vertex $u\in H$. By Lemma~\ref{distance}, if $M^+$ is unicyclic
then there is a homomorphism from $M^+$ to $H$ such that all
vertices in $U$ are mapped to $u$.

Obtain the graph $G'$ from $G$ by (i) removing all
of the vertices of distance less than $r$ from the  unique cycle
of $M$, and (ii) contracting
$U$ into a single new vertex $u$. Suppose that $h'$ is a
homomorphism from $G'$ to $H$. Then by the previous paragraph,
$h'$ can be extended to a homomorphism from $G$ to $H$ where
each vertex $v\in V(G')-u$ is mapped to $h'(v)$, and every vertex in $U$
is mapped to $h'(u)$. Thus, if $G'$ is homomorphic to $H$ then so is $G$.

We now define $T$ to be
the tree which consists of a vertex adjacent to $L$
leaves. Since the degree of $u$ in $G'$ is at most $L$,
and using the fact that all vertices of $M$ are deleted
when forming $G'$ (here is where we require $r>|M|$),  the rest now
follows as in the $k\geq3$ case.
%it is easy to see that again we have
%\begin{eqnarray*}
%\Pr[G^k_{n,p} \oplus M \mbox{ is not homomorphic to $H$}]
%&\le& \Pr[G^k_{n,p} \oplus T \mbox{ is not homomorphic to $H$} ] + \frac{\alpha}{2}\\
%&\leq&\Pr[G^k_{n,(1+\epsilon)p}\mbox{ is not homomorphic to $H$}] + \frac{\alpha}{2},
%\end{eqnarray*}
%which contradicts Corollary~\ref{toolcsp}(b).
\proofend

\subsection{Disconnected Graphs}
In this subsection we discuss the possibilities of extending
Theorem \ref{thomo} to the case where $H$ is disconnected. We will
focus on graphs, i.e. the $k=2$ case.

When considering disconnected graphs, it is helpful to note that if
the $H'$-homomorphism problem has a sharp threshold for every component $H'$ of $H$,
then the $H$-homomorphism problem has a sharp threshold. In fact,
it is simply the smallest of the thresholds for its components.
To see this, note first that for each component $H'$, since the $H'$-homomorphism
problem has a sharp threshold, every tree or unicyclic graph must have a homomorphism to
$H'$.  $G_{n,p=c/n}$ a.s. has at most one component with more than one cycle
(i.e. a giant component). So a.s. $G_{n,p}$ is homomorphic to $H$ iff either (i)
there is no giant component or (ii) the giant component
is homomorphic to $H$.  Since the giant component is connected, it is homomorphic
to $H$ iff it is homomorphic to at least one component of $H$. That giant component will a.s.
be homomorphic to at least one component of $H$ iff it is a.s. homomorphic
to the one with the smallest satisfiability threshold.

We now show that if there is at least one graph $H$ such that Hypothesis
A does not hold for the $H$-homomorphism problem,
then there must be such a graph with two components: a
triangle and a graph that is triangle-free and not
3-colourable.

Assume that Hypothesis A does not hold for $H$.  That is, every unicyclic
graph has a homomorphism to $H$, and the $H$-homomorphism problem has a coarse threshold.

First note that a triangle is not homomorphic to any triangle-free
graph. Also, Lemma \ref{cycle-hom}(b) says that every unicyclic graph is homomorphic to a
triangle. So ``every unicyclic graph has a homomorphism to $H$" is
equivalent to ``$H$ contains a triangle''.

Since the $H$-homomorphism problem has a coarse threshold,
there is some component $H_i$ of $H$, such that the
$H_i$-homomorphism problem has a coarse threshold. By Theorem \ref{thomo},
$H_i$ has no triangle.

Let $H^1$ be the subgraph of $H$ which consists of all triangle-free
components and $H^2$ be the subgraph consisting of the remaining
components of $H$; we have argued that neither $H^1$ nor $H^2$ is
empty. Since each component $H_i$ of $H^2$ has a triangle, Theorem
\ref{thomo} implies that the $H_i$-homomorphism problem has a sharp
threshold. Therefore, the $H^2$-homomorphism problem has a sharp
threshold.

Suppose that $H^1$ is
$3$-colourable. Then every graph homomorphic to $H^1$
is also homomorphic to a triangle, and hence is homomorphic to $H^2$.
It follows that being homomorphic
to $H$ is equivalent to being homomorphic to $H^2$. But this contradicts
the facts that $H$-homomorphism has a coarse threshold and $H^2$-homomorphism
has a sharp threshold. Therefore $\chi(H^1)>3$.

Also, we know that $H^1$-homomorphism has a coarse threshold since $H^1$ is triangle-free.

Let $c_3(n)$ denote the 3-colourability threshold, and let $c'(n)$
be the $H^2$-homomorphism threshold.  Every 3-colourable graph is
homomorphic to a triangle and thus is homomorphic to $H^2$.
Therefore $c'(n)\geq c_3(n)$. For every $c<c'(n)$, $G_{n,p=c/n}$
a.s. has a homomorphism to $H^2$ and hence to $H$.  Since
$H$-homomorphism has a coarse threshold, there must be some
$c=c(n)>c'(n)$ for which $G_{n,p=c/n}$ is not a.s.
non-$H$-homomorphic.   With probability bounded away from zero, the
non-giant components of $G_{n,p=c/n}$ are trees and hence are
homomorphic to $H^1$ and $H^2$. Therefore a.s. the giant component
of $G_{n,p=c/n}$ is not $H^2$-homomorphic as otherwise $G_{n,p=c/n}$
would not be a.s. non-$H^2$-homomorphic. Therefore the giant
component must not be a.s. non $H^1$-homomorphic  as otherwise
$G_{n,p=c/n}$ would be a.s. non-$H$-homomorphic. Therefore
$G_{n,p=c/n}$ is not a.s. non-$H^1$-homomorphic.

Replacing $H^2$ by a triangle has the effect of replacing $c'(n)$ by $c_3(n)$.
Since this does not increase $c'(n)$, we still have some $c^*=c^*(n)>c'(n)$ for which
$G_{n,p=c^*/n}$ is not a.s. non-$H^1$-homomorphic. If $H^1$ is disconnected,
add some edges so that it is connected but remains triangle-free; call the resulting
subgraph $(H^1)'$, and call the resulting graph, i.e. $(H^1)'$ plus a triangle component, $H'$.
Any graph that is
$H^1$-homomorphic is $(H^1)'$-homomorphic and so $G_{n,p=c^*/n}$ is not a.s. non-$(H^1)'$-homomorphic.
Since $H^1$ is triangle-free, Theorem \ref{thomo} implies that $(H^1)'$-homomorphism
has a coarse threshold.  The range of this threshold  either includes $c^*(n)$
or lies higher than $c^*(n)$; either way, it contains a range of values that is
higher than $c_3(n)$.  In that range of values of $c$, $G_{n,p=c^*/n}$ a.s. has no
homomorphism to the triangle component of $H'$.  It follows that that range of values
lies within the range of a coarse threshold for $H'$.

Therefore, $H'$ has a triangle but $H'$-homomorphism has a coarse threshold
and so violates Hypothesis A.

So the question of whether there is any undirected graph $H$ for which the
$H$-homomorphism problem violates Hypothesis A is equivalent to the following:

\begin{question}\label{q1} Is there any triangle-free graph $H_1$ with $\chi(H_1)>3$
such that for some values of $n$ and some $c>c_3(n)$, $G_{n,p=c/n}$
is not a.s. non-$H_1$-homomorphic, where $c_3(n)$ is the threshold value of
$3$-colorability?
\end{question}

\subsection{Directed Graphs}\label{sdg}
Earlier in this section, we discussed whether there exist any connected graphs $H$
for which the $H$-homomorphism problem violates Hypothesis A.  Now we turn our
attention to directed graphs.
We provide an example of a disconnected directed graph $H$ which
comes close to violating Hypothesis A.  It has the properties:
\begin{enumerate}
\item every unicyclic digraph has a homomorphism to $H$, and

\item the $H$-homomorphism problem under the $\csph$ model has a
coarse threshold.
\end{enumerate}

This does not actually violate Hypothesis A. The coarse
threshold is under the $\csph$ model rather than the $\csp$ model.

For a directed graph $D$, let $\tilde{D}$ denote the undirected
graph that is obtained from $D$ by removing the directions from the
edges and then replacing each double edge by a single edge. We define $D_{n,p}$ to be the random digraph on $n$
vertices where each of the $n(n-1)$ potential directed edges is
present with probability $p$.  Thus $D_{n,p}$ possibly
contains both edges $uv$ and $vu$ for some pair of vertices $v,u$,
i.e. a 2-cycle;  in fact, if $p=c/n$ for a constant $c$, then it
is straightforward to show that the probability that $D$ contains
at least one 2-cycle is $\z+o(1)$ for some constant $\z=\z(c)<1$.
This is the reason that we need to use the $\csph$ model rather than the $\csp$ model;
the digraphs formed by the $\csp$ model cannot have any 2-cycles since they
cannot have more than one constraint on the same pair of variables.

$H$ consists of a specific connected digraph $H_1$, defined below, and a
pair of vertices $u_1,u_2$, where the edges $u_1,u_2$ and
$u_2,u_1$ are both present.
%There is at least one edge from $H_1$ to $\{u_1,u_2\}$ (this edge is only needed
%for the connectivity of $H$).
There are no edges between
$\{u_1,u_2\}$ and $H_1$.
$H_1$ has the following properties:

\begin{itemize}
\item[(i):] every unicyclic digraph which does not contain a
2-cycle a.s. has a homomorphism to $H_1$; and
\item[(ii):]  For some $c_1>1/2$,
$D_{n,p=c_1/n}$ is not a.s. non-$H_1$-homomorphic.
\end{itemize}

It is easy to see that any unicyclic digraph, whose cycle is a
2-cycle, has a homomorphism to the 2-cycle.  By (i), every other
unicyclic digraph has a homomorphism to $H_1$.  Thus, every
unicyclic digraph has an $H$-homomorphism, as claimed. We will
show that, for every $1/2<c<c_1$, $D_{n,p=c/n}$ is neither a.s.
$H$-homomorphic nor a.s. non-$H$-homomorphic.  Thus, we have a
coarse threshold.  Condition (ii) above implies the latter, so we
just need to prove the former.

The graph $\tilde{D}$ for $D=D_{n,p}$, a.s. has a
giant component, since it is equivalent to $G_{n,p}$ where
$p=1-(1-\frac{c}{n})^2=\frac{2c}{n}-o(\inv{n})$ and $2c>1$. The number of 2-cycles
in $D$ are easily shown to have a Poisson distribution with constant mean,
using the same analysis as in, eg., Section 3.3 of \cite{jlr}.  Also, each edge of $\tilde{D}$ is
equally likely to correspond to a 2-cycle in $D$. Therefore, with probability bounded away from
zero, $D$ has a 2-cycle which lies in the giant component of $\tilde{D}$.
It is not hard to
see that a.s. if $D$ has such a 2-cycle then there is no $H$-homomorphism: That 2-cycle must
be mapped onto $u_1,u_2$. Since $H$ has no edges between $H_1$ and
$\{u_1,u_2\}$, any vertex that can be reached in $\tilde{D}$ from
that 2-cycle must also be mapped onto $u_1$ or $u_2$. So the
entire giant component must be mapped onto $\{u_1,u_2\}$. It is well-known that
a.s. a giant component is not 2-colourable and hence has an odd cycle.
(This follows, eg. from the facts that: (i) $\tilde{D}$ is a.s. not 2-colourable
and (ii) with probability at least some positive constant, all components of $\tilde{D}$
other than the giant one are trees and hence 2-colourable.)
Thus a.s. $\tilde{D}$ has an odd cycle and no odd cycle
can be mapped onto a 2-cycle. Therefore, $D$ is not a.s.
$H$-homomorphic.

It remains only to prove the existence of some $H_1$ satisfying
(i), (ii). We will choose $H_1$ to be a tournament (i.e. for every pair
of vertices, exactly one of the possible edges between them is
present) which contains every loopless 2-cycle-free directed graph on $k_0$ vertices as
a subgraph where $k_0$ is a constant defined below.  This can be done trivially if
$|H_1|\geq k_02^{k_0\choose2}$; simply place each tournament on $k_0$ vertices
on a different set of vertices of $H_1$.

For an undirected graph $G$ which does not contain any multiple
edges, the \emph{oriented chromatic number} $\chi_o$ of $G$ is the
minimum number $k$ such that every loopless 2-cycle-free directed graph $D$ satisfying
$\tilde{D}=G$ is homomorphic to a loopless 2-cycle-free directed graph $H$ with at most
$k$ vertices. The \emph{acyclic chromatic number} of a graph $G$
is the least integer $k$ for which there is a proper coloring of
the vertices of $G$ with $k$ colors in such a way that every cycle
of $G$ contains at least $3$ different colors. It was proved in
\cite{Raspaud} that if the acyclic chromatic number of a graph $G$
is at most $k$, then its oriented chromatic number is at most
$k\times2^{k-1}$. We also need the following:

\begin{lemma}
\label{acyclic} There exists a number $c>1$ such that a.s. the
acyclic chromatic of $G_{n,p=c/n}$ is at most $5$.
\end{lemma}
\begin{proof}
Let $G=G_{n,p}$. A \emph{pendant} path in $G$ is a path in which
no vertices other than the endpoints lie in any edge of the graph
off the path. It was proven in \cite{mm2} (see the proof of Lemma 6)
that there exists $c>1$ such that a.s.
after removing the internal vertices of pendant paths of length at
least $4$ from $G$ every component is either a tree or it is
unicyclic. One can use $3$ colors to color the vertices in these
components and then use $2$ other colors to color the removed
vertices such that every cycle in $G$ is colored by at least $3$
colors.
\end{proof}

When $D=D_{n,p=c_1/n}$ every edge is present in
$\tilde{D}$ with probability $2p-p^2$ and independent of the other
edges. Thus, if $c_1=c/2$, where $c$ is the constant
obtained from Lemma~\ref{acyclic}, a.s. the acyclic chromatic
number of $\tilde{D}$ is at most 5 and so $\chi_o(\tilde{D})\leq 5\times 2^4=k_0$.
Therefore, a.s. if $D$ is 2-cycle-free then $D$ is homomorphic to some
loopless 2-cycle-free digraph $H$ on at most $k_0$ vertices.  Since every such $H$ is a subgraph
of $H_1$, this would mean $D$ is homomorphic to $H_1$. Since $D$ does not a.s. have
a 2-cycle, this establishes that $H_1$ satisfies (ii).

\section{The $(d,k,t)$-model}\label{sdkt}

In this section, we prove Theorem \ref{tdkt} which says that
the $(d,k,t)$-model has a sharp threshold whenever $d,k\geq2$ and $1\leq t<d^{k-1}$;
i.e. whenever $d,k,t$ are not such that the model is a.s. unsatisfiable for all $c>0$.

{\bf Proof of Theorem \ref{tdkt}} Suppose that the $(d,k,t)$-model
exhibits a coarse threshold. Then consider $p,\a,\e$ and $M$ as
guaranteed by Corollary~\ref{toolcsp}. As in the proof of Theorem
\ref{thomo}, we will find some hypertree $T$ such that adding $T$ to
the random CSP increases the probability of unsatisfiability by a
constant. This time, adding $T$ may not boost the probability as
much as $M$ does; but it will boost it by a small constant, and that
will be enough.

Gent et.al.\cite{gmp} proved that for $k=2$:
if $t<d^{k-1}$ and $M$ is unicyclic, then $M$ is satisfiable
(see their Theorem 1 and Corollaries 1 and 2). Their argument easily
extends to the case where $k\geq3$. So we can assume that $M$ is satisfiable.
Suppose that $V(M)=u_1,...,u_r$ and let $a_i$ be the value of
$u_i$ in some particular satisfying assignment $A$ of $M$. Given a
CSP $F$ on at least $r$ variables, we define $F\oplus A$ to be the
CSP formed by choosing a random ordered $r$-tuple of variables
$v_1,...,v_r$ in $F$ and for each $1\leq i\leq r$, forcing $v_i$
to take the value $a_i$ by adding a one-variable constraint on
$v_i$. Clearly the probability that $\csph\oplus A$ is
unsatisfiable is at least as high as the probability that
$\csph\oplus M$ is unsatisfiable.

Lemma~\ref{local} implies that there exists some constant $L$ such that,
defining $E_1$ to be the event that ``every $v_i$ has at most $L$
neighbours and no hyperedge contains two $v_i,v_j$'',
$\pr(E_1)\geq 1-\frac{\a}{2}$. Suppose that $E_1$
holds.

Consider a particular $v_i$ and expose the $\ell\leq L$
edges containing it, $e_1,...,e_{\ell}$ and the corresponding
constraints $C_1,...,C_{\ell}$.  For each $C_j$, let $C'_j$ be the
$(k-1)$-variable constraint obtained by restricting $v_i$ to be
$a_i$; i.e. a $(k-1)$-tuple of values is permitted for $C'_j$ iff
$C_j$ permits that same $(k-1)$-tuple along with $v_i=a_i$.  Thus,
we can remove $C_j$ and add the constraint $C_j'$ on the $k-1$ other variables.
After doing so for every $C_j$, we can remove the restriction that $v_i=a_i$,
since $v_i$ no longer lies in any constraints. (Note that, since $E_1$ holds,
no constraint will contain some pair $v_i,v_j$ and thus be reduced twice.)

It is useful now to consider choosing $\csph\oplus A$ by first selecting
the random variables $v_1,...,v_r$ and then choosing $\csph$. Thus, carrying
out the operation described in the previous paragraph is equivalent to,
for each $v_i$: expose $\ell$, choose $\ell$ random $(k-1)$-tuples of
variables from $V-\{v_1,...,v_r\}$; for each selected $(k-1)$-tuple,
choose a random $C_j$ with exactly $t$ restrictions
and place $C_j'$ on the $(k-1)$-tuple; then remove $v_i$.
Note that since $C_j$ has $t$ restrictions, $C'_j$ has at most $t$ restrictions.

For convenience, we modify the experiment in a manner that increases
the probability of unsatisfiability.  For each $v_i$, instead of randomly exposing $\ell$,
we simply assume $\ell=L$; i.e. we choose $L$ random $(k-1)$-tuples.
Next, if $C_j'$ has fewer than $t$ restrictions, we add more restrictions
to it so that it has exactly $t$ restrictions.  Our final modification
is that instead of picking a random $(k-1)$-tuple and randomly selecting
$C_j'$ as described above, we choose ${d^{k-1}\choose t}$ random $(k-1)$-tuples
and place each of the ${d^{k-1}\choose t}$ possible
constraints on $k-1$ variables and with $t$ restrictions on one of the $(k-1)$-tuples.
The last modification may appear to be a bit of an overkill,
but it has the (minor) convenience that the added constraints are not
randomly selected.  Finally, we wish to do without the fact that $v_1,...,v_r$
are not permitted to be selected as members of the $(k-1)$-tuples.  So we choose the
$(k-1)$ tuples randomly from amongst all vertices and let $E_2$
be the event that none of them use any of
the vertices in $v_1,...,v_r$.  Since $r=O(1)$ and we are choosing a total
of $O(1)$ $(k-1)$-tuples, $\pr(E_2)=1-O(n^{-1})$.

So we let $G$ be a random CSP formed as follows: start with a random
$\csph$ and then for each of the ${d^{k-1}\choose t}$ possible
constraints on $k-1$ variables and with $t$ restrictions, choose
$rL$ random ordered $(k-1)$-tuples of variables and place that
constraint on them.
$\pr(G \mbox{ is unsatisfiable}|E_2)\geq \pr(\csph\oplus A\mbox{ is unsatisfiable }$,
as described in the preceding paragraph. Since $\pr(E_2)=1-o(1)$, this
implies that adding those ${d^{k-1}\choose t}rL$ constraints boosts the
probability of unsatisfiability by at least $2\a-o(1)$. We say that a
collection of ${d^{k-1}\choose t}rL$  $(k-1)$-tuples is
{\em bad} if adding the constraints to that set results in an
unsatisfiable CSP. So, consider the following random experiment:
pick a random $\csph$ and then pick ${d^{k-1}\choose t}rL$ ordered
$(k-1)$-tuples of the variables. The probability that we pick a
bad collection is at least $2\a$. Since ${d^{k-1}\choose t}rL=O(1)$, a
simple first moment calculation shows that a.s. the choice of
$(k-1)$-tuples will be vertex disjoint.  Thus, the probability of
picking a bad collection is at least $2\a-o(1)$ even if we condition on
the $(k-1)$-tuples being vertex-disjoint.

Now we are finally ready to define our hypertree $T$ as follows:
(i) take the hypergraph consisting of a vertex
$v$ lying in $rL{d^k\choose t}$ edges where no other vertex lies
in more than one of the edges (i.e. a star), and (ii) place each of the
${d^k\choose t}$ possible $(d,k,t)$-constraints on $rL$ of the
edges.  This, of course, is where we take advantage of the fact that
we are using the $(d,k,t)$-model.

Now consider adding a copy of $T$ to $\csph$.
For each $1\leq\d\leq d$, let $T_{\d}$ denote the
collection of $(k-1)$-tuples obtained by removing $v$ from every
edge of $T$ that contains a constraint in which every restriction has
$v=\d$; note that $T_{\d}$ consists of $rL$ copies of every constraint
on $k-1$ variables with exactly $t$ restrictions and so $|T_{\d}|={d^{k-1}\choose t}rL$.
The probability that
for each $1\leq\d\leq d,$ $T_{\d}$ is a bad set is at least
$(2\a-o(1))^d$.  Note that if every $T_{\d}$ is a bad set, then
the resulting CSP is unsatisfiable because setting $v=\d$ requires
the set of $(k-1)$-constraints on $T_{\d}$ to be enforced.

So by Lemma \ref{tree}, there is some $z=o(n^{1-k})$ such that
the probability that $\widehat{\CSP}_{n,p=p+z}({\cal P})$ is
unsatisfiable is at least $(2\a-o(1))^d$.  By considering adding
$x$ copies of $T$, we see that the probability that
$\widehat{\CSP}_{n,p=p+xz}({\cal P})$ is satisfiable is at most
$(1-(2\a-o(1))^d)^x$ which is less than $\a$ for some sufficiently
large constant $x$.  Since $z=o(n^{1-k})$, this implies that
$\pr(\widehat{\CSP}_{n,(1+\e)p}({\cal P})\mbox{ is
unsatisfiable})>1-\a$ which contradicts
Corollary~\ref{toolcsp}(b). \proofend

\section{Binary CSP's with domain size $3$}\label{s23}

Recall that the arguments from Istrate\cite{gi} and from Creignou
and Daud\'{e}\cite{cd2} can show that when the domain size $d=2$,
then Hypothesis A holds; i.e.  if every unicyclic CSP is
satisfiable, then $\csp$ has a sharp threshold.  This result does
not extend to $d=3$.  Consider the following example, with
$d=3,k=2$:

\begin{example}\label{ed3}
We have two constraints.  $C_1$ says that either both variables are equal
to 1, or neither is equal to 1.  $C_2$ says that the variables cannot both
have the same value.
$\calp(C_1)=\frac{2}{3},\calp(C_2)=\frac{1}{3}$.
\end{example}

Observe that every unicyclic CSP that uses only constraints
$C_1,C_2$ is satisfiable.

Consider any $\frac{3}{2}<c<3$. Thus, a.s. the sub-CSP formed
by the $C_1$ constraints has a giant component,
and the sub-CSP formed by the $C_2$ constraints does not.
We will show that $\csp$
is neither a.s. satisfiable nor a.s. unsatisfiable.

To see that it is not a.s. unsatisfiable, note that the subgraph
induced by the $C_2$ constraints is 2-colourable with probability
at least some positive constant. This follows from the well known
fact\cite{er} that for $c<1$ the random graph $G(n,\frac{c}{n})$ is
a forest with probability at least some positive constant. If
it is 2-colourable, then we can satisfy all the $C_2$ constraints
by assigning every variable either 2 or 3; this will not violate
any $C_1$ constraints.

To see that it is not a.s. satisfiable, note that the subgraph formed
by the $C_1$ constraints has a giant component $T$. So either
every variable in $T$ is assigned 1 or none of them are. A.s. at
least one $C_2$ constraint has both variables in $T$, and so they
cannot both be assigned 1. Thus, a.s. no variables in $T$ can be
assigned 1. This implies that if the $C_2$ constraints form an odd
cycle using variables of $T$ then the CSP is not satisfiable. That
event occurs with probability at least some positive constant, because the
graph formed by the $C_2$ constraints that use only variables of $T$
is $G_{n',\frac{1/3}{n'}}$ where $n'=|T|=\Theta(n)$ and
so it is not 2-colourable with probability at least some positive constant.

It is instructive to look at this example in light of Corollary
\ref{toolcsp}. Here, the subgraph $M$ is a triangle whose edges are
all $C_2$ constraints. If $M$ appears, then at least one of its
variables must be assigned the value 1. However, with probability at
least some $\z>0$, the remainder of the CSP has a structure implying
that if at least one of those variables has the value 1 then some
specific set of $\Theta(n)$ other variables all must be assigned the
value 1.  But a.s. there is a $C_2$ constraint between at least two
of those variables and thus they can't all be assigned 1. This
enables the appearance of $M$ to boost the probability of
unsatisfiability by at least some positive constant.

The main result of this section shows that when $d=3$ and $k=2$,
if Hypothesis A fails on some model, then there must be an $M$
as in Corollary \ref{toolcsp}
whose presence boosts the probability of unsatisfiability
for essentially the same reason as that in Example \ref{ed3}.

Before presenting our theorem, we begin with a few preliminaries.
To simplify, we will take $k=2$.

Consider a CSP where every constraint is on 2 variables.
Suppose there is some constraint on variables $v,u$ which
implies that if $v$ is assigned $\d$ then $u$ must be assigned $\g$;
we say that
$v:\d\rightarrow u:\g$. Moreover if there is a sequence
of variables $v_1,\ldots,v_r$ and values
$\d_1,\ldots,\d_r$ such that $v_i:\d_i\rightarrow v_{i+1}:\d_{i+1}$
for $i=1,\ldots,r-1$ then we say that $v_1:\d_1\rightarrow v_r:\d_r$.

%\begin{theorem}\label{t23} Consider some $\calp$ with
%$d=2,k=3$ such that every unicyclic CSP formed from
%$\supp(\calp)$ is satisfiable.  $\csp$ has a coarse threshold
%iff there exists a unicyclic CSP $M$
%formed from $\supp(\calp)$,
%a value $1\leq\d\leq3$, $p=p(n),\e>0,z>0,b>0$ such that:
%\begin{enumerate}
%\item[(a)] $\e<\pr(\csp\mbox{ is satisfiable })<1-\e$;

%\item[(b)] $M$ cannot be satisfied using only the two values other
%than $\d$;

%\item[(c)] $\csph$ a.s. has at least $z n$ variables $v$ such that
%$v:\d\rightarrow u:\d$ for at least $b n$ variables $u$.
%\end{enumerate}
%\end{theorem}

%Thus, this explains the only ways that a model with $d=3,k=2$
%can have a coarse threshold.

For each variable $v$ and each pair of (possibly equal) values $\d,\g$
we define $F_{\d,\g}(v)$ to be the set of variables $u$ such that
$v:\d\rightarrow u:\g$, and we define $F_{\d}(v)=\cup_{1\leq\g\leq d}F_{\d,\g}(v)$.
Thus $F_{\d}(v)$ is the set of variables $u$ such that if $v$ is assigned $\d$
then there is {\em a path of constraints} which imply that $u$ must be assigned
a particular value.  Assigning $\d$ to $v$  may force assignments to
other variables $w$ via a combination of more than one path of constraints.
But the locally tree-like nature of $\csp$ will imply that such variables
are not a significant concern.

In $\csp$, we can expose $F_{\d}(v)$ by using a simple breadth-first
search from $v$.  This allows us to analyze the distribution of the
size of $F_{\d}(v)$ and $F_{\d,\g}(v)$ using a standard
branching-process analysis (see e.g.\ Chapter 5 of \cite{jlr}).
Straightforward branching-process arguments yield:

\begin{lemma}\label{l231} With the exception of at most $d^2$ constants $c$,
if $p=(c+o(1))/n$ then for every pair $\d,\g$,
one of these cases holds:
\begin{enumerate}
\item[(a)] $\ex(|F_{\d,\g}(v)|)=O(1)$; or
\item[(b)] $\ex(|F_{\d,\g}(v)|)=\Theta(n)$.
\end{enumerate}
\end{lemma}

We omit the standard proofs.  We remark only that those $d^2$ constants  are
the so-called {\em critical points} for each $F_{\d,\g}$.

We say that $F_{\d,\g}$ is {\em subcritical} if case (a) holds and {\em supercritical}
if case (b) holds. We say that $F_{\d}$ is {\em supercritical}
if $F_{\d,\g}$ percolates for at least one $\g$, and is {\em subcritical} otherwise.
Markov's Inequality immediately implies:

\begin{lemma}\label{l232}
\begin{enumerate}
\item[(a)] If $F_{\d,\g}$ is subcritical, then for every $\xi>0$
there is a constant $L$ such that\\ $\pr(|F_{\d,\g}(v)|\leq L)>1-\xi$.
\item[(b)] If $F_{\d,\g}$ is supercritical, then there are constants $\z,\b>0$
such that $\pr(|F_{\d,\g}(v)|\geq \b n)>\z$.
\end{enumerate}
\end{lemma}

%WHERE DO WE USE (b)????

We can now state the main result of this section:

\begin{theorem}\label{t23} Consider some $\calp$ with
$d=2,k=3$ such that  $\csp$ has a coarse threshold and every unicyclic CSP formed from
$\supp(\calp)$ is satisfiable. Then there exist $p,\a,\e,M$
as in Corollary \ref{toolcsp} such that for some value
$1\leq\d\leq3$ we have:
\begin{enumerate}
\item[(a)] $M$ cannot be satisfied using only the two values other
than $\d$; and
\item[(b)] $F_{\d,\d}$ is supercritical.
\end{enumerate}
\end{theorem}

E.g., in Example \ref{ed3} we can take $M$ to be a triangle of $C_2$
constraints and $\d=1$.

This theorem does not
extend to $d=4,k=2$ nor $d=3,k=3$; in each of these cases, we have counterexamples.
Before proving our theorem, we start with a helpful lemma.

%{\bf Remark:} It is straightforward to verify, using similar reasoning to
%that for Example \ref{ed3} that conditions (a,b) of Theorem \ref{t23} imply condition (c)
%of Corollary \ref{toolcsp}. We provide a proof at the end of this
%section.

%We say that $F_{\d,\g}$ is {\em giant} if there is some $b,z>0$ such that
%a.s. there are at least $z n$ variables $v$ with $|F_{\d,\g}(v)|\geq b n$.

%\begin{lemma}\label{lg}
%If $F_{\g,\g}$ is supercritical then $F_{\g,\g}$ is giant.
%\end{lemma}

\begin{lemma}\label{l233} Suppose $p=(c+o(1))/n$ where $c$ is not one
of the 9 exceptional constants from Lemma \ref{l231}. If $F_{\d}$ is supercritical then
either (i) $F_{\d,\d}$ is supercritical or (ii) there is some $\m$ such that $F_{\m,\m}$
is supercritical and there is a sequence of constraints in $\supp(\calp)$
through which $v:\d\rightarrow u:\m$.
\end{lemma}

{\bf Proof:}  If $F_{\d,\d}$ is supercritical then (i) holds;
so assume otherwise.  Thus there is some $\g\neq\d$ such that $F_{\d,\g}$
is supercritical. Thus there is a sequence of constraints in $\supp(\calp)$
through which $v:\d\rightarrow u:\g$.  So if $F_{\g,\g}$ is supercritical
then (ii) holds with $\m=\g$; so assume otherwise.

Let $\a$ be the third value. Consider any vertex $v$ and any constant $\z>0$.

{\em Case 1: there is no sequence of constraints in $\supp(\calp)$
through which $v:\d\rightarrow u:\a$.}
 Thus $F_{\d,\a}(v)=\emptyset$.
By Lemma \ref{l232}(a) there is some constant $L_1$ such
that with probability at least $1-\z/3$, $|F_{\d,\d}(v)|\leq L_1$.  If that bound holds,
then by Lemma \ref{local}, there is a constant $L_2$ such that
with probability at least $1-\z/3$, the set of variables $u$ such that
there is some $w\in F_{\d,\d}(v)$ and some constraint implying
$w:\d\rightarrow u:\g$ has size at most $L_2$; call the set of such variables $X$.
If that second bound holds,
then applying Lemma \ref{l232}(a) again, there is a constant $L_3$ such that
with probability at least $1-\z/3$, $|\cup_{x\in X}F_{\g,\g}(x)|<L_3$.
Therefore, with probability at least $1-\z$, $|F_{\d,\g}(v)|<L_3$; since this
holds for every $\z>0$, this contradicts Lemma \ref{l232}(b) and the fact that
$F_{\d,\g}$ is supercritical.

{\em Case 2: there is a sequence of constraints in $\supp(\calp)$
through which $v:\d\rightarrow u:\a$.} Then the same argument that showed
$F_{\g,\g}$ is subcritical shows that $F_{\a,\a}$ is also subcritical.
We proceed as in Case 1:

By Lemma \ref{l232}(a) there is some constant $L_1$ such
that with probability at least $1-\z/5$, $|X|\leq L_1$.
If that bound holds, then by Lemma \ref{local}, there is a constant $L_2$ such that
with probability at least $1-\z/5$, the set of variables $u$ such that
there is some $w\in F_{\d,\d}(v)$ and some constraint implying
$w:\d\rightarrow u:\a$ has size at most $L_2$; call the set of such variables $X_1$.
If that second bound holds,
then applying Lemma \ref{l232}(a) again, there is a constant $L_3$ such that
with probability at least $1-\z/5$, $|\cup_{x\in X_1}F_{\a,\a}(x)|<L_3$.
If that bound holds, then again by Lemma \ref{local}, there is a constant $L_4$ such that
with probability at least $1-\z/5$, the set of variables $u$ such that either (i)
there is some $w\in F_{\d,\d}(v)$
and some constraint implying $w:\d\rightarrow u:\g$
or (ii) there is some $w\in\cup\left( \cup_{x\in X_1}F_{\a,\a}(x)\right)$
and some constraint implying $w:\a\rightarrow u:\g$
has size at most $L_4$; call the set of such variables $X_2$.
Finally, if all those bounds hold,
then applying Lemma \ref{l232}(a) again, there is a constant $L_5$ such that
with probability at least $1-\z/5$, $|\cup_{x\in X_2}F_{\g,\g}(x)|<L_5$.
Therefore, with probability at least $1-\z$, $|F_{\d,\g}(v)|<L_5$; since this
holds for every $\z>0$, this contradicts Lemma \ref{l232}(b) and the fact that
$F_{\d,\g}$ is supercritical.
\proofend

We are now ready to prove our theorem.

{\bf Proof of Theorem \ref{t23}:}
Suppose that $\csph$ has a coarse threshold and consider
$M,\e,\a,p=p(n)$ from Corollary~\ref{toolcsp}. If $p=(c+o(1))/n$ where
$c$ is one of the 9 exceptional constants from Lemma \ref{l231}, then
we can increase $p$ by some small $\e'/n$, and decrease $\e$ slightly, so that
the conditions of Corollary~\ref{toolcsp} still hold.  This allows us
to apply Lemmas \ref{l231}, \ref{l233}.

As defined in \cite{mm1}, a value $1\leq \d\leq 3$ is {\em bad} if there
is a constraint in $\supp(\calp)$ which forbids a variable from receiving $\d$; i.e. if
there is a constraint $C$ which contains the restrictions $(\d,1),(\d,2),(\d,3)$
or the restrictions $(1,\d),(2,\d),(3,\d)$. A value $\d$ is also said to be
bad if there is a sequence of constraints in $\supp(\calp)$ joining variables $u,v$ for which
$v:\d\rightarrow u:\g$ where $\g$ is a bad value. It is easy to see that if there is a unicyclic
CSP $M$ formed from the constraints of $\supp(\calp)$ such that every satisfying
assignment to $M$ uses at least one bad value, then $M$ can be modified to an unsatisfiable unicyclic
CSP $M$' formed from $\supp(\calp)$: we simply attach paths of constraints to
each variable of $M$ which forbid those variables from receiving any bad
values. (See \cite{mm1} for the details.) Therefore, if every unicyclic CSP formed
from $\supp(\calp)$ is satisfiable, then every such CSP can be satisfied without using any
bad values. Thus, $M$ can be satisfied without using any bad values.

{\em Case 1: There is a value $\d$ such that (i)
every satisfying assignment of $M$ must use $\d$ or a bad value
on at least one variable and (ii)
$F_{\d}$ is supercritical.}

If there are any bad values, then the above construction produces a
unicyclic $M'\supset M$ such that every satisfying assignment of $M'$
must use $\d$ on at least one variable. (Otherwise, set $M'=M$.)
If $F_{\d,\d}$ is supercritical
then $M',\d$ satisfy Theorem \ref{t23}.  Otherwise, by Lemma \ref{l233},
there is a value $\m\neq\d$ such that $F_{\m,\m}$
is supercritical and there is a sequence of constraints so that $v:\d\rightarrow u:\m$.
Attaching that sequence to every variable of $M'$ yields a unicyclic CSP $M''$
for which every satisfying assignment must use $\m$ on at least one variable.
Thus  $M'',\m$ satisfy Theorem \ref{t23}.

{\em Case 2: There is a satisfying assignment $A$ of $M$
in which every value $\d$ used is such that $\d$ is not bad and
$F_{\d}$ is not supercritical.}   Suppose
that $M$ has $r$ variables $x_1,...,x_r$ and that $A$ assigns $a_i$ to $x_i$.
Recall from Section \ref{sdkt} that $\csp\oplus A$ is formed by taking
$\csp$ and then choosing $r$ random variables $v_1,...,v_i$ and adding
one-variable constraints that force $v_i$ to take $a_i$.  Clearly
$\pr(\csp\oplus A\mbox{ is unsatisfiable})\geq\pr(\csp\oplus M\mbox{ is unsatisfiable})$.

Expose $F=\cup_{i=1}^rF_{a_i}(v_i)$, and $U$, the set of variables
outside of $F$ that lie in a constraint with a variable in $F$.
Since all of the $F_{a_i}$ are subcritical, Lemmas \ref{local} and \ref{l232}(a)
imply that there is some $L$ such that with probability at least $1-\a/2$,
$|U|<L$ and $F\cup U$ is a forest with $r$ trees, one containing each $a_i$.
Since adding $M$ to $\csph$ increases
the probability of unsatisfiability by at least $\a$, it must be that the probability
that $\csph$ is satisfiable, $|U|\leq L$, $\F\cup U$ is such a forest
and $\csph\oplus A$ is unsatisfiable
is at least $\a/2$.

Suppose that $\csph$ is satisfiable, $|U|\leq L$ and $F\cup L$ is a
forest of $r$ treees, one for each $a_i$. Note that the forest
structure of $F\cup U$ implies that every variable whose value
is determined by the assignment $A$ lies in $F$.  Consider some $u\in U$
sharing a constraint with $w\in F$ where $A$ forces $w$ to take the value $\m$.
Let $\Omega=\Omega(u)$ be the set of values which can be assigned to $u$ which,
in conjunction with assigning $\m$ to $w$ do not violate their constraint.
We know that $|\Omega|\neq0$ since otherwise $\m$ is a bad value and hence
some $a_i$ is a bad value.
We know that $|\Omega|\neq1$ since otherwise $u\in F$. So
$|\Omega(u)|\geq 2$ for each $u\in U$. Suppose that $u_1,...,u_\ell$
are the variables in $U$ with $|\Omega|=2$,
and let $\d_i$ be the value not in $\Omega(u_i)$. Consider taking
a random CSP formed as follows:  first take a $\csph$ and then
choose $\ell$ random variables $u_1,...,u_{\ell}$ and force $u_i$ to not
take value $\d_i$ using a one-variable constraint.
We have proved that the one-variable constraints
boost the probability of unsatisfiability by at least $\a/2$.

We have now reached a stage where the rest of the proof is by now
standard. We can use the techniques in any of \cite{ef,af,fk} to
show that if $\ell$ 1-variable constraints, each of the form ``$v_i$
cannot receive $\d_i$" boost the probability of unsatisfiability by
at least $\a/2$, then so does the addition of some constant number
of additional random constraints. This will lead to a contradicton
of Corollary \ref{toolcsp}(c).  We will take the most concise of
these techniques, the one from \cite{fk} (which was proposed by
Alon). The main tool is a variant of a theorem of Erd\H{o}s and
Simonovits \cite{es},  as stated in \cite{ef3}:

\begin{lemma}\label{les} For all positive integers $k,\ell$ and real $0<\g\leq 1$,
there exists $\g'>0$ such that for sufficiently large $n$, if $H\subseteq [n]^{\ell}$
is such that $|H|\geq\g n^{\ell}$ then with probability at least $\g'$ a random choice
of $\ell$ $k$-tuples of integers between 1 and $n$: $(v_1^1,...,v^k_1),...,
(v_{\ell}^1,...,v_{\ell}^k)$ yields a complete $\ell$-partite system of elements of $H$; i.e. for
every function $f:[\ell]\rightarrow[k]$, the $\ell$-tuple
$(v_1^{f(1)},...,v_{\ell}^{f(\ell)})\in H$.
\end{lemma}

To apply this lemma, we set $k=2$ and we
let $H$ be the set of {\em bad} $\ell$-tuples of
variables in $\csp$; i.e. those $\ell$-tuples $v_1,...,v_{\ell}$ such that if
we forbid each $v_1$ from receiving $\d_i$ then the CSP will be unsatisfiable.
We have shown that choosing a random $\ell$-tuple $v_1,...,v_{\ell}$ and
forbidding each $v_i$ from receiving $\d_i$ boosts the probability of unsatisfiability
by at least $\a/2$.  That random choice will only make the CSP unsatisfiable if
we choose a bad $t$-tuple of variables.  Therefore, $|H|\geq (\a/2) n^{\ell}$,
and so Lemma \ref{les} applies to $H$.

Now suppose that instead of adding those $\ell$ 1-variable constraints
to $\csp$, we instead add $\ell$ new random constraints selected according to $\calp$;
call these constraints $C_1,...,C_{\ell}$.
For each value $\d$, there is at least one constraint $C_{\d}\in\supp(\calp)$
which does not allow both variables to recieve $\d$; otherwise $\csp$ would
be trivially satisfiable by setting every variable equal to $\d$.
The probability that for each $1\leq i\leq \ell$, $C_i=C_{\d_i}$ is
$\prod_{i=1}^{\ell}\calp(C_{\d_i})=\z>0$.  If this event occurs,
then we can treat each $C_i$ as a pair (i.e. 2-tuple) of variables
at least one of which cannot take the value $\d_i$.  By Lemma \ref{les},
with probability at least some $\g'>0$ this collection of $\ell$ pairs
forms a complete $\ell$-partite system of elements of $H$.  If so, then
the resulting CSP is unsatisfiable: To see this, consider any satisfying
assignment and for each $1\leq i\leq \ell$, set $\phi(i)$ to be a variable of
$C_i$ which does not have the value $\d_i$.  Then the $\ell$-tuple
$(\phi(1),...,\phi(\ell))$ is a member of $H$ and hence is bad.
Thus there is no satisfying assignment in which each $\phi(i)$ does
not receive $\d_i$.

Therefore, adding $\ell$ random constraints increases the probability
of unsatisfiability by at least $\z\g'>0$.  So adding a sufficiently
large constant number of additional random constraints will boost it by
arbitarily close to 1.  Increasing $p$ by $\e/n$ will a.s. result in
at least that many extra constraints.  So this contradicts Corollary
\ref{toolcsp}(c). This establishes Claim 4 and hence our Lemma.
\proofend

We close this section by noting why this proof cannot be extended to
general $d$.  The problem is that possibly some of the variables in
$U$ would have their domain sizes reduced by two instead of one and
so some of the 1-variable constraints would be of the form ``$v_i$
cannot receive $\d_i$ or $\g_i$''. This would prevent us from using
Lemma \ref{les} and any of the other known techniques for
establishing sharp thresholds.

\section{Future Directions}
There is clearly much work still to be done along these lines of research.
The big problem still remains - determine precisely
which models from \cite{mm1} have a sharp threshold.
Of course, Section 2 indicates that this may be overly ambitious.
%But lowering our sights only slightly, we can try to determine
%all possible {\em causes} for coarse thresholds, i.e. continue the
%course started in Section \ref{s23}.
In the example of Section 2, $\supp(\calp)$ is {\em disconnected}
in that the values can be partitioned into two parts (namely $\{1,2,3\}$ and $\{4,5\}$
such that no constraint permits its variables to take members of different parts.
In \cite{mm2}, it was noted that when $\supp(\calp)$ is disconnected $\csp$ can
behave strangely.  So perhaps it is more feasible to determine precisely which
models with $\supp(\calp)$ connected have sharp thresholds.
An important subgoal would
be to do this for binary CSP's, i.e. the case where $k=2$.
Another reasonable goal to pursue would be the $d=3$ case.

As far as more specific classes of models go, one should try to
extend the work in Section \ref{sh} and examine
whether Hypothesis A holds for $H$-homomorphism problems
when $H$ is a {\em directed} hypergraph.  Such homomorphism
problems are equivalent to CSP's in which every constraint
is identical under some permutation of the variables.
Of course, we showed in Section \ref{sdg} that this is not always true
in the $\csph$ model.  But there is a chance that it is true
for the $\csp$ model.  Also, the example in Section \ref{sdg} is not connected.
So perhaps Hypothesis A holds for
$H$-homomorphism problems whenever $H$ is a connected directed hypergraph.
Or perhaps one needs to require that $H$ is strongly connected.
%We don't have a strong guess as to whether Theorem \ref{thomo}
%extends to the directed case.  It seems that there is a very good
%chance that the ``connected'' condition may have to be replaced
%by ''strongly connected''; i.e. it might be that there are examples
%of connected directed hypergraphs $H$ for which Hypothesis A fails,
%but no such examples that are strongly connected.
And of course,
it would be good to determine whether the ``connected'' condition can
be removed from Theorem \ref{thomo} by answering Question \ref{q1}.

\bibliographystyle{plain}
\bibliography{csp2}
\end{document}